\pdfoutput=1 %may be required for arXiv
\documentclass[11pt]{amsart}

\usepackage{epsfig,overpic,comment}
\usepackage[usenames,dvipsnames,svgnames,table]{xcolor}
\usepackage[pagebackref,linktocpage=true,colorlinks=true,linkcolor=Blue,citecolor=BrickRed,urlcolor=RoyalBlue]{hyperref}
\usepackage[msc-links,abbrev]{amsrefs} %shortalphabetic,backrefs
\usepackage{amsmath,amsthm,amssymb,marginnote}
\usepackage{enumerate}

\newtheorem{thm}{Theorem}[section]

\newtheorem{lem}[thm]{Lemma}
\newtheorem{prop}[thm]{Proposition}

\theoremstyle{definition}
\newtheorem{note}[thm]{Note}

\newcommand{\be}{\begin{equation}}
\newcommand{\ee}{\end{equation}}
\newcommand{\ol}{\overline}
\newcommand{\R}{\mathbf{R}}

\newcommand{\C}{\mathcal{C}}

\renewcommand{\epsilon}{\varepsilon}
\renewcommand{\S}{\mathbf{S}}
\renewcommand{\tilde}{\widetilde}

\DeclareMathOperator{\inte}{int}

\DeclareMathOperator{\conv}{conv}

\DeclareMathOperator{\cl}{cl}

\title[Shortest closed curve to contain a sphere]{Shortest closed curve to contain a sphere\\ in its convex hull}
\author{Mohammad Ghomi}
\address{School of Mathematics, Georgia Institute of Technology,
Atlanta, GA 30332}
\email{ghomi@math.gatech.edu}
\urladdr{www.math.gatech.edu/~ghomi}

\author{James Wenk}
\address{School of Mathematics, Georgia Institute of Technology,
Atlanta, GA 30332}
\email{jwenk3@math.gatech.edu}
\urladdr{www.math.gatech.edu/~jwenk3}

\date{\today \,(Last Typeset)}
\thanks{Research of M.G. was supported in part by NSF Grant DMS--2202337.}

\begin{document}
\vspace*{-0.75in}

\maketitle

\begin{abstract}
We show that in Euclidean 3-space  any closed curve which contains the unit sphere  within its convex hull has length $L\geq4\pi$, and characterize the case of equality. This result generalizes the authors' recent solution to a  conjecture of  Zalgaller. Furthermore, for the analogous problem in $n$ dimensions, we  include the estimate $L\geq Cn\sqrt{n}$ by Nazarov,  which is sharp up to the constant $C$.
\end{abstract}

%\tableofcontents

\section{Introduction}\label{sec:intro}
The \emph{convex hull} of a set $X$ in Euclidean space $\R^3$ is the intersection of all convex sets which contain $X$. The \emph{inradius} of $X$ is the supremum of the radii of spheres which are contained in $X$. Here we show:

\begin{thm}\label{thm:main}
Let $\gamma\colon [a,b]\to\R^3$ be a closed rectifiable curve of length $L$, and $r$ be the inradius of the convex hull of $\gamma$. Then
\begin{equation}\label{eq:main}
L\geq 4\pi r.
\end{equation}
Equality holds only if, up to a reparameterization, $\gamma$ is  simple, $\C^{1,1}$, lies on a sphere of radius $\sqrt2 \,r$, and traces consecutively $4$ semicircles of length $\pi r$.
\end{thm}

 In 1996 V. A. Zalgaller \cite{zalgaller:1996,orourkeMO} conjectured that the above theorem holds subject to the additional assumption that  $\gamma$ lie outside a sphere $S$ of radius $r$ within its convex hull. The length minimizer, called the \emph{baseball curve}, together with $S$, is shown in Figure \ref{fig:baseball}.
 \begin{figure}[h]
\begin{overpic}[height=1.25in]{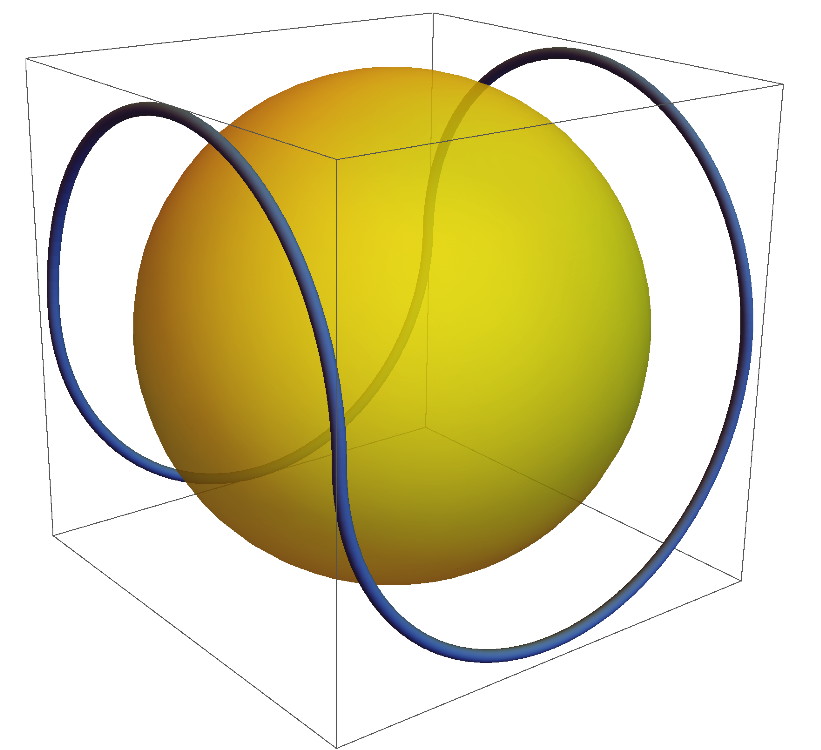}
\end{overpic}
\caption{The baseball curve}\label{fig:baseball}
\end{figure}
Zalgaller's conjecture was proved recently in \cite{ghomi-wenk2021} following earlier work in \cite{ghomi:lwr}. Here we refine the methods introduced in those papers to establish the more general result above.
Our approach will be similar to that in \cite{ghomi-wenk2021}. We start by setting $r=1$ and assuming that $\gamma$ has the smallest length among closed curves which contain the unit sphere $\S^2$ within their convex hull \cite[Sec 2.]{ghomi-wenk2021}. The \emph{horizon} of  $\gamma$ is the measure in $\S^2$ counted with multiplicity of the set of points $p\in\S^2$ where the affine tangent plane $T_p\S^2$ intersects $\gamma$:
$$
H(\gamma):=\int_{p\in\S^2} \#\gamma^{-1}(T_p\S^2)\, dp.
$$
Since $\gamma$ is closed, one quickly sees that $\#\gamma^{-1}(T_p\S^2)\geq 2$ for almost every $p\in\S^2$ \cite[Lem. 7.1]{ghomi:lwr}. Hence $H(\gamma)\geq 8\pi$. The \emph{efficiency} of  $\gamma$ is given by 
$$
E(\gamma):=\frac{H(\gamma)}{L(\gamma)}.
$$
So to establish \eqref{eq:main} it suffices to show that $E(\gamma)\leq 2$. To this end we note that for any partition of  $\gamma$ into subcurves $\gamma_i$,  
$$
E(\gamma)= \sum_i \frac{H(\gamma_i)}{L(\gamma)}=\sum_i \frac{L(\gamma_i)}{L(\gamma)} E(\gamma_i).
$$
So  it suffices to construct a partition with $E(\gamma_i)\leq 2$.  Similar to \cite{ghomi-wenk2021}, this is achieved by \emph{unfolding} $\gamma$ into the plane (Section \ref{sec:unfolding}), and identifying a collection of subcurves of $\gamma$ we call \emph{spirals} (Section \ref{sec:spirals}); however, these operations need to be generalized here as they were defined only for curves with $|\gamma|\geq 1$ in \cite{ghomi-wenk2021}. Furthermore, we will show that if $E(\gamma)=2$, then $|\gamma|\geq 1$. So the  rigidity of \eqref{eq:main}  follows from Zalgaller's conjecture established in \cite{ghomi-wenk2021}, and completes the proof of Theorem \ref{thm:main} (Section \ref{sec:proof}). 

For curves in $\R^2$ the isoperimetric inequality quickly yields $L\geq 2\pi r$ as the analogue of \eqref{eq:main}.
We will include in the \nameref{sec:appendix} a version of \eqref{eq:main} by F. Nazarov for curves in $\R^n$, which is obtained by covering the unit sphere $\S^{n-1}$ with certain slabs, and applying the correlation inequality \cite{royden2014, latala-matlak2017} to their Gaussian volume. 
This approach has implications for covering problems for the sphere by congruent disks \cite{boroczky-wintsche2003}, and yields a new proof of a result of Tikhomirov \cite{tikhomirov2015} (Note \ref{note:tikhomirov}).
There are many natural optimization problems for convex hull of space curves which remain open, including other questions of Zalgaller \cite{zalgaller:1996} which are closely related to well-known problems  of Bellman \cite{bellman:1956,bellman1963,ahks2022} in operations research and search theory \cite{alpern-gal2003,gal2013}; see also \cite{ghomi:lwr, ghomi-wenk2021, nikonorov2022} and references therein.

\section{Minimal Inspection Curves}\label{sec:minimal}
$\R^n$ denotes the $n$-dimensional Euclidean space with  inner product $\langle\cdot,\cdot\rangle$, norm $|\cdot|:=\langle\cdot,\cdot\rangle^{1/2}$, and origin $o$. 
A \emph{curve} is a continuous rectifiable mapping $\gamma\colon[a,b]\to\R^n$ with length $L=L(\gamma)$. We also use $\gamma$ to refer to its image $\gamma([a,b])$. If $\gamma(a)=\gamma(b)$ then we say that $\gamma$ is \emph{closed} and  identify $[a,b]$ with the topological circle $\R/(b-a)$.
Rectifiable curves may be parameterized with constant speed \cite{bbi:book}, which we assume is the case throughout this work.  In particular all curves below are Lipschitz continuous, and thus differentiable almost everywhere, with $|\gamma'|=L/(b-a)$; see \cite[Sec. 2]{ghomi-wenk2021} and references therein for basic facts on rectifiable curves.    We say  $\gamma$ is a \emph{(generalized) inspection curve} provided that $\gamma$ is closed and  its convex hull, $\conv(\gamma)$, contains the unit sphere $\S^2$.  It follows from Arzela-Ascoli theorem that there exists an  inspection curve $\gamma$ whose length achieves the minimum value among all inspection curves \cite[Sec. 2]{ghomi-wenk2021}. Then $\gamma$ will be called a 
\emph{minimal}  inspection curve.  We let  \emph{int}, \emph{cl}, and $\partial$, stand respectively for interior, closure, and boundary.

\begin{lem}\label{lem:segments1}
Let $\gamma\colon\R/L\to\R^3$ be a minimal  inspection curve. Suppose that $\gamma(t)\in \inte(\conv(\gamma))$, for some $t\in\R/L$. Then there exists a connected open set $U\subset\R/L$, with $t\in U$, such that $\gamma$ maps $\cl(U)$ injectively to a  line segment with end points on $\partial\conv(\gamma)$. In particular, $\gamma(t)=o$ for at most finitely many $t\in\R/L$.
\end{lem}
\begin{proof}
Let $U$ be the component of $\gamma^{-1}(\inte(\conv(\gamma)))$ which contains $t$. If $\gamma|_{\cl(U)}$ does not trace a line segment, we may shorten $\gamma$ by replacing $\gamma(\cl(U))$ with the line segment connecting the end points of $\gamma(\cl(U))$. But this operation preserves $\conv(\gamma)$, as it preserves the points of $\gamma$ on $\partial\conv(\gamma)$. Hence we obtain an inspection curve shorter than $\gamma$,  which is impossible. If $\gamma(t)=o$, then $L(\gamma|_U)\geq 2$, since $\gamma(U)$ contains a diameter of $\S^2$. So there can be only finitely many such points, since $\gamma$ is rectifiable. 
\end{proof}

We say that $t$ is a \emph{regular} point of a curve $\gamma$ provided that $\gamma$ is differentiable at $t$ and $\gamma'(t)\neq 0$. Then the \emph{tangent line} of $\gamma$ at $t$ is well defined. Since we assume that curves are parameterized with constant speed, they are regular almost everywhere. Furthermore, by Lemma \ref{lem:segments1}, all points $t\in\R/L$ with $\gamma(t)\in\inte(\conv(\gamma))$ of a minimal inspection curve $\gamma$ are regular.

\begin{lem}\label{lem:segments2}
Let $\gamma\colon\R/L\to\R^3$ be a minimal  inspection curve,  $t\in\R/L$ be a regular point of $\gamma$, and $\ell$ be the tangent line of $\gamma$ at $t$. Suppose that $\ell$ intersects $\inte(\conv(\gamma))$. Then there exists an open interval $U\subset\R/L$,  with $t\in U$, which is mapped injectively by $\gamma$ into $\ell\cap\inte(\conv(\gamma))$.
\end{lem}
\begin{proof}
If $\gamma(t)\in\partial\conv(\gamma)$, then either $\gamma'(t)$ or $-\gamma'(t)$ points outside $\conv(\gamma)$.
Hence, for some $s$ close to $t$, $\gamma(s)$ lies outside $\conv(\gamma)$, which is impossible. So $\gamma(t)\in\inte(\conv(\gamma))$, in which case Lemma \ref{lem:segments1} completes the proof.
\end{proof}

Combining the last two observations we obtain: 

\begin{prop}\label{prop:o}
Let $\gamma\colon\R/L\to\R^3$ be a minimal inspection curve. Then there exists an open set $U\subset\R/L$ such that  tangent lines of $\gamma$ on $U$ do not pass through $o$. Furthermore if  $U\neq\R/L$, then $\R/L\setminus U$ is the disjoint union of a finite number of closed intervals  each mapped by $\gamma$ into a line segment which passes through $o$ and ends on $\partial\conv(\gamma)$. 
\end{prop}
\begin{proof}
Let $X$ be the union of all closed intervals $I\subset \R/L$ such that $\gamma(I)$ is a line segment which passes through $o$ and ends on $\partial\conv(\gamma)$. By Lemma \ref{lem:segments1}, there are at most finitely many such intervals. Thus $X$ is closed. Let $U:=\R/L\setminus X$. By Lemma \ref{lem:segments2}, no tangent line of $\gamma$ at a regular point of $U$ may pass through $o$, which completes the proof.
\end{proof}

\section{Unfolding}\label{sec:unfolding}

Let $\gamma\colon \R/L\to\R^3$ be a minimal  inspection curve. We will always assume that $0$ is a local minimum point of $|\gamma|$.
By Lemma \ref{lem:segments1}, $\gamma$ passes through $o$ at most finitely many times which, if they exist,  
will be denoted by $0=:t_0, \dots, t_m:=L$. Then the projection $\ol\gamma\colon\R/L\to\S^2$, given by $\ol\gamma:=\gamma/|\gamma|$ is well defined on $\R/L\setminus \{t_k\}$. Furthermore since, by Proposition \ref{prop:o}, $\gamma$ traces line segments near $t_k$, $\ol\gamma$ is Lipschitz on each interval $(t_{k-1},t_k)$.
Thus $\ol\gamma$ is differentiable almost everywhere on $\R/L$. Consequently,  the  arclength function 
$$
\theta(t):=\int_0^t|\ol\gamma'(s)|ds
$$
is well defined on $[0,L]$  ($\theta$ measures the ``cone angle" \cite{cks2002} or ``vision angle" \cite{choe-gulliver1992} of $\gamma$ from the point of view of $o$).
The \emph{unfolding}   of $\gamma$ is  the planar curve $\tilde\gamma\colon[0,L]\to\R^2$ defined as 
$$
\tilde\gamma(t):=|\gamma(t)|e^{i\big(\theta(t)+(k-1)\pi\big)},\quad\text{for}\quad t\in[t_{k-1},t_k].
$$
Note that $|\gamma|=|\tilde\gamma|$, and whenever $\gamma$ passes through $o$, then $\tilde\gamma$ will pass through $o$ as well on a line segment.
As in \cite{ghomi-wenk2021},  we may also compute that
\be\label{eq:theta-prime}
|\tilde\gamma'|=\big||\gamma|'+i|\gamma|\theta'\big|,\quad\quad\text{and}\quad\quad 
\theta'=|\ol\gamma'|=\frac{1}{|\gamma|^2}\sqrt{|\gamma|^2|\gamma'|^2-\langle\gamma,\gamma'\rangle^2},
\ee
almost everywhere. It follows that, for almost all $t\in[0,L]$,
$
|\tilde\gamma'|=|\gamma'|=1.
$
So $\tilde\gamma$ is parameterized by arclength, and $L(\gamma)=L(\tilde\gamma)$. 
Hence, by \cite[Cor. 3.2]{ghomi-wenk2021}, 
$
E(\gamma)=E(\tilde\gamma)
$
since points of $\gamma$ with $|\gamma|\leq 1$ make no contribution to $E(\gamma)$.
Furthermore, the angles $\alpha:=\angle(\gamma, \gamma')$ and  $\tilde\alpha:=\angle(\tilde\gamma, \tilde\gamma')$ are defined almost everywhere, and
\begin{equation}\label{eq:alpha}
\alpha
=\cos^{-1}\left(|\gamma|'\right)
=\cos^{-1}\left(|\tilde\gamma|'\right)
=\tilde\alpha.
\end{equation}

\begin{lem}\label{lem:1-1}
Let $\gamma\colon\R/L\to\R^3$ be a minimal inspection curve. Then
$\tilde\gamma$ is locally one-to-one.
\end{lem}
\begin{proof}
Let $U$ be as in Proposition  \ref{prop:o}. Then $\gamma$ and $\gamma'$ are linearly independent at all regular points of $U$. So \eqref{eq:theta-prime} shows that $\theta'>0$ almost everywhere on $U$, via Cauchy-Schwarz inequality. Hence $\theta$ is strictly increasing on $U$, which yields that $\tilde\gamma$ is star-shaped with respect to $o$ in a neighborhood of each point of  $U$. Since, by Proposition  \ref{prop:o}, $\gamma$ traces a line segment on each component of $\R/L\setminus U$, $|\gamma|$ is strictly monotone on each of these components. Hence $\tilde\gamma$ is one-to-one on each component of $\R/L\setminus U$, since $|\tilde\gamma|=|\gamma|$. Finally,  $\tilde\gamma$ is one-to-one in a neighborhood of each point of $\partial U$, since $\tilde\gamma$ is locally star-shaped on $U$ and it maps each component of $\R/L\setminus U$ to a line passing through $o$.
\end{proof}

A planar curve $\gamma\colon[a,b]\to\R^2$ is \emph{locally convex} provided that it is locally one-to-one and each point $t\in [a,b]$ has a neighborhood $U\subset[a,b]$ such that $\gamma(U)$ lies on the boundary of a convex set. A \emph{side} of a  line $\ell\subset\R^2$ is one of the two closed half spaces determined by $\ell$. A \emph{local supporting line} $\ell$ for $\gamma$ at $t$ is a line passing through $\gamma(t)$ with respect to which  $\gamma(U)$ lies on one side. If $\gamma(U)$ lies on a side of $\ell$ which contains $o$, then we say that $\ell$ lies \emph{above} $\gamma$. 
Finally, if $\gamma$ is locally convex and through each point of it there passes a local support line which lies above $\gamma$, then  we say that 
$\gamma$ is locally convex \emph{with respect to $o$}. Note that if $\gamma$ is locally convex with respect to $o$ and passes through $o$, then $\gamma$ must trace a line segment near $o$. 

\begin{lem}\label{lem:loc-convex}
Let $\gamma\colon \R/L\to\R^3$ be a minimal inspection curve. Then
$\tilde\gamma$  is locally convex with respect to $o$.
\end{lem}
\begin{proof}
Let $U$ be as in Proposition  \ref{prop:o}, and $t\in U$. By Lemma \ref{lem:1-1}, there exists a neighborhood $V$ of $t$ in $U$ on which $\tilde\gamma$ is one-to-one. Furthermore, $\tilde\gamma(V)$ is star-shaped with respect to $o$. So connecting the end points of $\tilde\gamma(V)$ to $o$ by line segments yields a simple closed curve. It is shown in the proof of \cite[Prop. 4.3]{ghomi-wenk2021} that this curve bounds a convex set, due to minimality of $\gamma$. Thus $\tilde\gamma$ is locally convex with respect to $o$ on $U$. Next suppose that $t\in\partial U$, and let $V$ be a small neighborhood of $t$ in $\cl(U)$. By Proposition \ref{prop:o},  $\tilde\gamma$ connects one end point of $\tilde\gamma(V)$ to $o$ by tracing a line segment. Connect the other end point of $\tilde\gamma(V)$ to $o$ by another line segment. Then the resulting simple closed curve again bounds a convex set by the argument in the proof of \cite[Prop. 4.3]{ghomi-wenk2021}. So $\tilde\gamma$ is locally convex with respect to $o$ on $\cl(U)$. Finally, $\tilde\gamma$ is locally convex with respect to $o$ on the complement of $\cl(U)$, since these regions are mapped to line segments, by Proposition  \ref{prop:o}.
\end{proof}

\section{Spiral Decomposition}\label{sec:spirals}
If $\gamma\colon[a,b]\to\R^2$ is a locally convex curve, parameterized with constant speed, then its one sided derivatives, $\gamma'_{\pm}$, are well-defined everywhere and are nonvanishing \cite[Lem. 5.1]{ghomi-wenk-arXiv2020}. Set $\gamma'(a):=\gamma'_+(a)$.
We say that $\gamma:[a,b] \to \R^2$ is a \emph{(generalized) spiral} provided that (i) $\gamma$ is locally convex with respect to $o$, (ii) $|\gamma|$ is nondecreasing, and (iii) $\langle\gamma(a),\gamma'(a)\rangle=0$. A spiral is  called \emph{strict}  if $|\gamma|$ is increasing. A \emph{spiral decomposition} of a curve $\gamma\colon[a,b]\to\R^2$ is a collection $U_i$ of mutually disjoint open subsets of $[a,b]$ such that (i)  $\gamma|_{\cl(U_i)}$ is a strict spiral, after switching the direction of $\gamma|_{\cl(U_i)}$ if necessary, and (ii) $|\gamma|'=0$ almost everywhere on $[a,b]\setminus \cup_i \cl(U_i)$. 

\begin{lem}\label{lem:GeneralizedDecomposition}
Let $\gamma\colon\R/L\to\R^3$ be a  minimal  inspection curve.  Then  $\tilde\gamma$ admits a spiral decomposition.
\end{lem}
\begin{proof}
The argument follows the same outline as in \cite[Prop. 5.2]{ghomi-wenk2021}, with minor modifications.
Recall that we assume  $0$ is a local minimum point of $|\gamma|$. If $|\gamma(0)|> 0$, then it follows that $\tilde\alpha(0)=\tilde\alpha(L)=\pi/2$. Otherwise, $|\tilde\gamma(0)|=|\tilde\gamma(L)|=0$, since $|\gamma|=|\tilde\gamma|$. Let $X$ be the set of points $t\in[0,L]$ such that $\tilde\gamma$ has a local support line at $\tilde\gamma(t)$ which is orthogonal to $\tilde\gamma(t)$, or $|\tilde\gamma(t)|=0$. Then $0$, $L\in X$ and $|\tilde\gamma|'=0$ almost everywhere on $X$.
Also note that $X$ is closed, since the limit of any sequence of support lines of a convex body is a support line, and the set of points with $|\tilde\gamma(t)|=0$ is compact. Consequently each  component $U$ of  $[0,L]\setminus X$ is an open subinterval of $[0,L]$. It remains to show that $\tilde\gamma|_{\cl(U)}$ is a spiral. By Lemma \ref{lem:loc-convex}, $\tilde\gamma|_{\cl(U)}$ is locally convex with respect to $o$. Furthermore, as argued in the proof of \cite[Prop. 5.2]{ghomi-wenk2021}, $|\tilde\gamma|'$ is always positive or always negative at differentiable points of $|\tilde\gamma|$ on $U$. So we may suppose that $|\tilde\gamma|$ is increasing on $U$, after switching the direction of $\tilde\gamma|_{\cl(U)}$ if necessary. Finally, let $x\in\partial U$ be the initial point of $\tilde\gamma|_{\cl(U)}$. If $|\tilde\gamma(x)|=0$, then $\tilde\gamma|_{\cl(U)}$ is a spiral. If $|\tilde\gamma(x)|>0$, it follows that $\tilde\gamma(x)$ it orthogonal to $\tilde\gamma'_+(x)$, which again shows that
$\tilde\gamma|_{\cl(U)}$ is a spiral and completes the proof.
 \end{proof}

Let $\S^1$ denote the unit circle in $\R^2$. The last observation quickly yields:

\begin{lem}\label{lem:GeneralizedDecomposition2}
Let $\gamma$, $\tilde\gamma$ be as in Lemma \ref{lem:GeneralizedDecomposition} and $\sigma\colon[a,b]\to\R^2$ be a spiral in the decomposition of $\tilde\gamma$. Let $t\in[a,b]$ be a regular point of both $\sigma$ and $\gamma$, and $\ell$ be the tangent line of $\sigma$ at $t$. Suppose that $\ell$ crosses $\S^1$. Then $\sigma([a,t])$ lies on $\ell$.
\end{lem}
\begin{proof}
Let $\ol\ell$ be the tangent line of $\gamma$ at $t$. If $\ell$ crosses $\S^1$, then $\ol\ell$ crosses $\S^2$, by \eqref{eq:alpha}. In particular, $\ol\ell$ intersects the interior of $\conv(\gamma)$. Then Lemma \ref{lem:segments2} completes the proof.
\end{proof}

The key point in the proof of Theorem \ref{thm:main} is:

\begin{prop}\label{prop:Esigma}
Let $\sigma\colon[a,b]\to\R^2$ be a spiral in the unfolding of a minimal inspection curve.  Then $E(\sigma)\leq 2$. Furthermore, if $|\sigma(a)|<1$, then $E(\sigma)<2$.
\end{prop}
\begin{proof}
If $|\sigma(a)|\geq 1$, then $E(\sigma)\leq 2$ by \cite[Prop. 2.7]{ghomi-wenk2021}. So we assume $|\sigma(a)|<1$.
We may also assume that $|\sigma(b)|>1$ for otherwise $H(\sigma)=0$ which yields $E(\sigma)=0$. Let $b'$ be the supremum of  $t\in[a,b]$ such that $\sigma([a,t])$ is a line segment. By Lemma \ref{lem:segments1}, $|\sigma(b')|\geq 1$. We may assume that $\sigma(a)$ lies on the nonnegative portion of the $y$-axis, and $\sigma([a,b'])$ lies to the right of the $y$-axis, see Figure \ref{fig:spirals}.
\begin{figure}[h]
\begin{overpic}[height=1.3in]{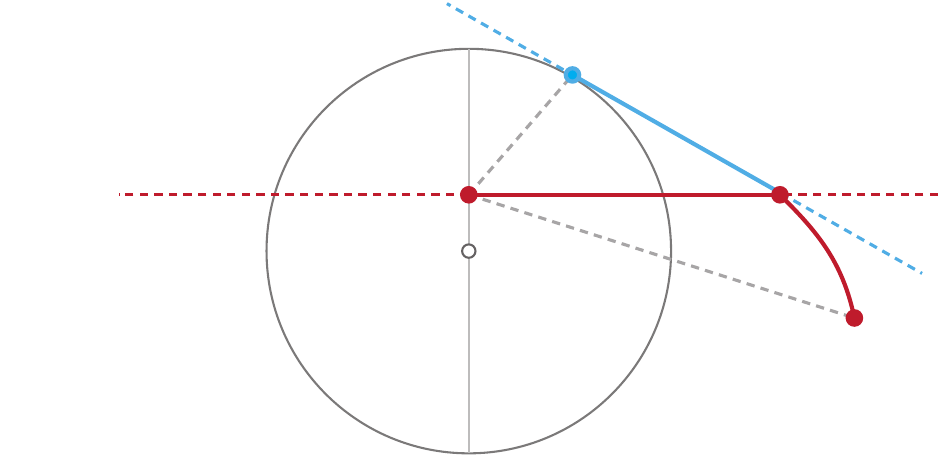}
\put(38.5,30.5){\small$\sigma(a)$}
\put(85,30.5){\small$\sigma(b')$}
\put(90,11){\small$\sigma(b'')$}
\put(60.5,43.5){\small$x$}
\put(9,27){\small$\lambda$}
\end{overpic}
\caption{Construction of the competing curve}\label{fig:spirals}
\end{figure}
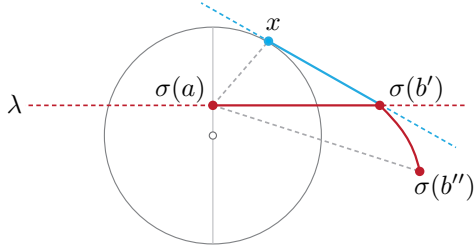
If $b'< b$, then we may choose $b'<b''<b$ such that $\sigma([b',b''])$ is convex, and lies to the right of the $y$-axis. Since $\sigma$ is locally convex with respect to $o$, 
$\sigma([b',b''])$ lies below the line $\lambda$ spanned by $\sigma([a,b'])$, if $|\sigma(a)|>0$. If $|\sigma(a)|=0$, we may still assume that $\sigma([b',b''])$ lies below $\lambda$ after a reflection.
 Consider the line which passes through $\sigma(b')$ and is tangent to the upper half of $\S^1$, say at a point $x$. Let $\tau$ be the curve obtained by joining the line segment $x\sigma(b')$  to the beginning of $\sigma|_{[b',b]}$. We will show that (i) $\tau$ is a spiral, and (ii) $E(\sigma)<E(\tau)$. Then we are done, because $E(\tau)\leq 2$ since its initial height is $\geq 1$.

First we check that $\tau$ is a spiral. This is obvious if $b'=b$. So assume that $b'<b$, and let $b'<b''<b$ be as defined above. It suffices to check that $\tau$ is locally convex at $\sigma(b')$. Connect the end points of the portion $x\sigma(b'')$ of $\tau$ to $\sigma(a)$ to obtain a closed curve $\Gamma$. Note that $\Gamma$ is simple since  $x\sigma(b')$ lies above $\lambda$ while $\sigma([b',b''])$ lies below it. Let $\theta$ be the interior angle of $\Gamma$ at $\sigma(b')$. We need to show that $\theta\leq\pi$. To this end let $t_i\in (b',b'')$ be a sequence of regular points of $\sigma$ converging to $b'$, and $\ell_i$ be tangent lines of $\sigma$ at $t_i$. Then $\ell_i$ converge to a support line of $\sigma([b',b''])$ at $\sigma(b')$, which we call $\ell$. By Lemma \ref{lem:GeneralizedDecomposition2}, $\ell_i$ do not cross $\S^1$. Consequently $\ell$ does not cross $\S^1$ either. So $\ell$ also supports $x\sigma(b')$. Hence $\ell$ is a support line of $\Gamma$ at $\sigma(b')$, which yields that $\theta\leq\pi$ as desired.

It remains to check that $E(\sigma)<E(\tau)$. To see this consider the triangle $\sigma(a)x\sigma(b')$. The interior angle of this triangle at $x$ is $\geq\pi/2$, 
since $\sigma(a)$ lies on the nonnegative portion of the $y$-axis. Hence $|x\sigma(b')|<|\sigma(a)\sigma(b')|$, which yields $L(\tau)<L(\sigma)$. On the other hand,  tangent planes of $\S^2$ intersect $\R^2\simeq\R^2\times\{0\}\subset\R^3$ in lines which do not cross $\S^1$, and any such line has exactly the same number of transverse intersections with $\sigma$ as it does with $\tau$. Hence $H(\tau)=H(\sigma)$ by definition of horizon. So $E(\sigma)<E(\tau)$ as desired.
\end{proof}

\section{Proof of Theorem \ref{thm:main}}\label{sec:proof}
Set $r=1$ and let $\gamma\colon\R/L\to\R^3$ be a minimal inspection curve, as discussed in Section \ref{sec:minimal}. To establish \eqref{eq:main} it suffices to show then that $E(\gamma)\leq 2$, as outlined in Section \ref{sec:intro}. In Section \ref{sec:unfolding} we established that $E(\gamma)=E(\tilde\gamma)$ where  $\tilde\gamma\colon[0,L]\to\R^2$ is the unfolding of $\gamma$.  By Lemma \ref{lem:GeneralizedDecomposition}, $\tilde\gamma$ admits a spiral decomposition, generated by a collection of mutually disjoint open sets $U_i\subset[0,L]$, $i\in I$. 
Set $U_0:=[0,L]\setminus \cup_i\cl(U_i)$, and let $\tilde\gamma_i:=\tilde\gamma|_{\cl(U_i)}$,  $\tilde\gamma_0:=\tilde\gamma|_{U_0}$. As in the proof of Zalgaller's conjecture in \cite[Sec. 10]{ghomi-wenk2021}, we have
\be\label{eq:E-gamma}
E(\tilde\gamma)=\frac{H(\tilde\gamma)}{L(\tilde\gamma)}=\frac{1}{L(\tilde\gamma)}\sum_iH(\tilde\gamma_i)=
\frac{1}{L(\tilde\gamma)}\left(L(\tilde\gamma_0)E(\tilde\gamma_0)+ \sum_iL(\tilde\gamma_i)E(\tilde\gamma_i)\right).
\ee
By Lemma \ref{lem:segments2}, every point $t\in[0,L]$ with $|\gamma(t)|<1$ lies on a line segment in $\gamma$ with end points on $\S^2$, and thus $\tilde\gamma(t)$ belongs to a strict spiral (with origin of the spiral corresponding to the midpoint of that line segment). So $|\tilde\gamma_0|\geq 1$. Then, as described in \cite[Sec. 10]{ghomi-wenk2021},  
$
E(\tilde\gamma_0)\leq 2.
$
Furthermore
$
E(\tilde\gamma_i)\leq 2
$
for all $i$
by Proposition  \ref{prop:Esigma}.
So $E(\tilde\gamma)\leq 2$ by \eqref{eq:E-gamma}, as desired. To characterize the case of equality in \eqref{eq:main}, note that by \eqref{eq:E-gamma}, if $E(\tilde\gamma)= 2$ then $E(\tilde\gamma_i)=2$. Consequently, by Proposition \ref{prop:Esigma},  $|\tilde\gamma_i|\geq 1$. So $|\tilde\gamma|\geq 1$, which yields $|\gamma|\geq 1$.  Hence, by the proof of Zalgaller's conjecture \cite[Thm. 1.1]{ghomi-wenk2021}, $\gamma$ is the baseball curve.

\section*[Appendix]{Appendix: Higher Dimensions}\label{sec:appendix}

Here we establish a higher dimensional version of \eqref{eq:main} due to Fedor Nazarov:

\begin{thm}[Nazarov]\label{thm:nazarov}
Let $\gamma\colon[a,b]\to\R^n$ be a  curve of length $L$, and $r$ be the inradius of the convex hull of $\gamma$.  Then 
\be\label{eq:nazarov}
L\geq Cn\sqrt{n}\,r,
\ee
 where $C>0$ is an absolute constant.
\end{thm}

By \emph{absolute constant} here we mean that $C$ does not depend on $n$ or $\gamma$. A Hamiltonian path in the edge graph of the \emph{cross polytope}, i.e., the unit ball with respect to the $L^1$-norm in $\R^n$, gives an example of a curve with $L\leq 2n\sqrt{2n}r$ \cite{ahks2022}. Thus \eqref{eq:nazarov} is sharp up to the constant $C$. 
To establish \eqref{eq:nazarov}, we may set $r=1$. Furthermore, we may assume that $n$ is even. Indeed suppose that \eqref{eq:nazarov}  holds for even $n$. If $n$ is odd and bigger than $1$, then we may project $\gamma$ into $\R^{n-1}$ to obtain $L\geq C(n-1)^{3/2}\geq (C/2) n^{3/2}$. Finally, it is enough to show that if  $L\leq C n\sqrt{n}$, for some absolute constant $C$,  then the inradius of $\conv(\gamma)\leq 1$,  which means that there exists $u\in\S^{n-1}$ such that $\langle  \gamma(t), u\rangle\leq 1$ for all $t\in [a,b]$. Equivalently, if $L\leq 2n\sqrt{n}$, then $\langle  \gamma(t), u\rangle\leq C/2$. In summary, it suffices to show:

\begin{prop}\label{prop:middle}
Let $\gamma\colon [a,b]\to\R^{2n}$ be a  curve of length $\leq 2n\sqrt{n}$. Then there exists $u\in\S^{2n-1}$ such that $\langle  \gamma(t), u\rangle\leq C$ for all $t\in [a,b]$.
 \end{prop}

To prove the above proposition,  we again assume that $\gamma$ has constant speed. Let $t_i\in[a,b]$, $i=0,\dots,n$, be equidistant points with $t_0:=a$, $t_{n}:=b$, and set $s_i:=(t_{i-1}+t_i)/2$ for $i=1,\dots, n$.  Let $H$ be an $n$-dimensional subspace of $\R^{2n}$ which is orthogonal to each $\gamma(s_i)$, and $\ol\gamma$ be the projection of $\gamma$ into $H$. Then $\ol\gamma|_{[t_{i-1}, s_i]}$, $\ol\gamma|_{[s_i, t_{i}]}$ are curves of length $\leq\sqrt{n}$ with one end at $o$, since $\gamma$ has constant speed.  So, identifying $H$ with $\R^n$, we have reduced Proposition \ref{prop:middle} to:

\begin{prop}\label{prop:2n}
Let $\gamma_i\colon[a,b]\to\R^n$, $i=1,\dots,2n$, be curves of length $\leq\sqrt{n}$ with $\gamma_i(a)=o$. Then there exists $u\in\S^{n-1}$ such that $\langle  \gamma_i(t), u\rangle\leq C$ for all $t\in[a,b]$.
\end{prop}

To prove the last proposition we employ the standard Gaussian measure, which is defined for Borel sets $A\subset\R^n$ as
$$
\mu(A):=\frac{1}{(\sqrt{2\pi})^n}\int_A e^{-|x|^2/2}\;d\lambda(x),
$$
where $\lambda$ is the $n$-dimensional Lebesgue measure. We also record that if $K_i$ are a family of convex sets which are symmetric with respect to $o$, then
\be\label{eq:gaussian}
\mu\left(\bigcap_i K_i\right)\geq \prod_i \mu(K_i)
\ee
by the Gaussian correlation inequality  \cite{royden2014, latala-matlak2017}. Here we need this fact only for slabs, which had been established in \cite{sidak1967}.

\begin{proof}[Proof of Proposition \ref{prop:2n}]
We set $[a,b]=[0,1]$ and assume that $\gamma_i$ have constant speed.
For every $t\in [0,1]$ and $i$ there exist vectors $v_{ik}(t)\in\R^n$,  such that 
$$
\gamma_i(t):=\sum_{k=1}^\infty v_{ik}(t), \;\;\quad\text{and}\quad\;\; |v_{ik}(t)|\leq \frac{\sqrt n}{2^k}.
$$
To generate these vectors, set $t_0:=0$, and let $t_k:=t_{k-1}-1/2^k$, if $t< t_{k-1}$, and $t_k:=t_{k-1}+1/2^k$ otherwise. Then we set $v_{ik}(t):=\gamma_i({t_{k}})-\gamma_i(t_{k-1})$. Note that each $v_{ik}(t)$ is chosen from a set $V_{ik}$, of cardinality $2^{k-1}$, which is independent of $t$. Now consider the slabs
$$
S(v):=\left\{\,x\in\R^n \,\middle\vert\; \big|\langle x,v\rangle\big|\leq\frac{\sqrt{n}}{k^2}\,\right\},\quad v\in V_{ik},
$$
which have width $2(\sqrt{n}/k^2)/|v|\geq 2(2^k/k^2)$, 
and set
$$
A:=\bigcap_{i=1}^{2n}\,\;\bigcap_{k=1}^\infty\bigcap_{v\in V_{ik}} S(v).
$$
By Fubini's theorem, and a standard estimate for the Gaussian integral,
$$
\mu\big(S(v)\big)
\geq
\frac{1}{\sqrt{2\pi}}\int_{-a_k}^{a_k} e^{-t^2/2}\,dt
\geq
1-e^{-a_k^2/2},
$$
where $a_k:=2^k/k^2$.
So by \eqref{eq:gaussian}, 
\begin{eqnarray*}
\mu(A)
\geq
\prod_{i=1}^{2n}\;\prod_{k=1}^\infty \prod_{v\in V_{ik}}\mu\big(S(v)\big)
\geq
\left(\prod_{k=1}^\infty\left(1-e^{-a_k^2/2}\right)^{2^{k-1}}\right)^{2n}.
\end{eqnarray*}
Since $\ln(1-e^{-x})\geq -2e^{-x}$ for $x\geq 32/81$, which is the smallest value of $a_k^2/2$ (achieved for $k=3$), we have
\begin{eqnarray*}
\prod_{k=1}^\infty \left(1-e^{-a_k^2/2}\right)^{2^{k-1}}
&=&
\exp\left(\sum_{k=1}^\infty 2^{k-1}\ln\left(1-e^{-a_k^2/2}\right)\right)\\
&\geq&
\exp\left(-\sum_{k=1}^\infty 2^ke^{-a_k^2/2}\right)\;\;=:\;\;\sqrt{\delta}\;\;>\;\;0.
\end{eqnarray*}
So we conclude that $\mu(A)\geq\delta^n$ where $\delta>0$ is an absolute constant.
Next note that, if $B_r^n$ is the ball of radius $r$ centered at $o$ in $\R^n$, with volume $|B_r^n|$, then
$$
\mu(B^n_r)
\le  
\frac{|B^n_r|}{(\sqrt{2\pi})^{n}}
=
\left(\frac{\sqrt{e}\, r}{\sqrt n}\right)^n   \frac{|B^n_{\sqrt n}|}{ (\sqrt{2\pi})^{n} (\sqrt{e})^{n}}
\le  
\left(\frac{\sqrt{e}\,r}{\sqrt n}\right)^n  \mu(B^n_{\sqrt n})
\le  
\left(\frac{\sqrt{e}\, r}{\sqrt n}\right)^n.
$$
So  if  $r:=\delta\sqrt{n}/\sqrt e$, then $\mu(B^n_r)\leq\delta^n \leq \mu(A)$. Consequently, $A\not\subset\inte(B^n_r)$ which means that there exists $u_0\in A$ with $|u_0|\geq r$. Now setting $u:=u_0/|u_0|$, we have
$$
\langle\gamma_i(t), u\rangle
=
\sum_{k=1}^\infty\langle v_{ik},u\rangle
\leq 
\frac{1}{r}\sum_{k=1}^\infty\langle v_{ik},u_0\rangle
\leq
\frac{\sqrt e}{\delta\sqrt{n}}\sum_{k=1}^\infty\frac{\sqrt{n}}{k^2}\leq\frac{2\sqrt e}{\delta}=:C,
$$
as desired.
\end{proof}

\begin{note}\label{note:tikhomirov}
When $\gamma_i$ in Proposition \ref{prop:2n} trace lines segments,  we obtain the following result in discrete geometry:
if $N\leq 2n$ points in $\R^n$ contain $\S^{n-1}$ within their convex hull, then at least one of them has distance $\geq \sqrt{n}/C$ from $o$. Equivalently, if $N\leq 2n$ disks of geodesic radius $\rho$ cover $\S^{n-1}$, then  $\cos(\rho)\leq C/\sqrt{n}$, which had been observed earlier by Tikhomirov \cite{tikhomirov2015}. Furthermore, proof of Proposition \ref{prop:2n} allows an estimate for $C$ as follows. If $\gamma_i$ trace line segments, we may set $k=1$. Then $\mu(S(v))\geq \int_{-2}^{2}e^{-t^2/2}dt/\sqrt{2\pi}\geq 0.95$. So  $\delta=(0.95)^2$, which yields $C=\delta/(2\sqrt{e})\simeq 3.65$. It has been conjectured that the optimal value of $C$ is $1$, which would correspond to the case where the  points form the vertices of a cross polytope \cite[Conj.1.3]{boroczky-wintsche2003}. This has been shown only for 
$n=3$ \cite{toth1943},  see \cite[p. 34]{toth1964}, and $n=4$ \cite{dlmz2000}.
\end{note}

%\addtocontents{toc}{\protect\setcounter{tocdepth}{0}}%for hiding from table of contents

\section*{Acknowledgments}
We thank Fedor Nazarov for communicating the proof of Theorem \ref{thm:nazarov} to us, and Galyna Livshyts for useful discussions and interest in this problem. Thanks also to the anonymous referee for suggestions to improve the exposition of this work.

%\addtocontents{toc}{\protect\setcounter{tocdepth}{1}}%to resume listing of the sections  
\bibliography{references}

% \bib, bibdiv, biblist are defined by the amsrefs package.
\begin{bibdiv}
\begin{biblist}

\bib{alpern-gal2003}{book}{
      author={Alpern, Steve},
      author={Gal, Shmuel},
       title={The theory of search games and rendezvous},
      series={International Series in Operations Research \& Management
  Science},
   publisher={Kluwer Academic Publishers, Boston, MA},
        date={2003},
      volume={55},
        ISBN={0-7923-7468-1},
      review={\MR{2005053}},
}

\bib{ahks2022}{article}{
      author={Antoniadis, Antonios},
      author={Hoeksma, Ruben},
      author={Kisfaludi-Bak, S\'{a}ndor},
      author={Schewior, Kevin},
       title={Online search for a hyperplane in high-dimensional {E}uclidean
  space},
        date={2022},
        ISSN={0020-0190},
     journal={Inform. Process. Lett.},
      volume={177},
       pages={Paper No. 106262, 4},
  url={https://doi-org.prx.library.gatech.edu/10.1016/j.ipl.2022.106262},
      review={\MR{4388498}},
}

\bib{bellman:1956}{article}{
      author={Bellman, Richard},
       title={A minimization problem},
        date={1956},
     journal={Bulletin of the AMS},
      volume={62},
       pages={270},
}

\bib{bellman1963}{article}{
      author={Bellman, Richard},
       title={An optimal search},
        date={1963},
     journal={Siam Review},
      volume={5},
      number={3},
       pages={274},
}

\bib{boroczky-wintsche2003}{incollection}{
      author={B\"{o}r\"{o}czky, K\'{a}roly, Jr.},
      author={Wintsche, Gergely},
       title={Covering the sphere by equal spherical balls},
        date={2003},
   booktitle={Discrete and computational geometry},
      series={Algorithms Combin.},
      volume={25},
   publisher={Springer, Berlin},
       pages={235\ndash 251},
  url={https://doi-org.prx.library.gatech.edu/10.1007/978-3-642-55566-4_10},
      review={\MR{2038476}},
}

\bib{bbi:book}{book}{
      author={Burago, Dmitri},
      author={Burago, Yuri},
      author={Ivanov, Sergei},
       title={A course in metric geometry},
      series={Graduate Studies in Mathematics},
   publisher={American Mathematical Society, Providence, RI},
        date={2001},
      volume={33},
        ISBN={0-8218-2129-6},
      review={\MR{1835418 (2002e:53053)}},
}

\bib{cks2002}{article}{
      author={Cantarella, Jason},
      author={Kusner, Robert~B.},
      author={Sullivan, John~M.},
       title={On the minimum ropelength of knots and links},
        date={2002},
        ISSN={0020-9910},
     journal={Invent. Math.},
      volume={150},
      number={2},
       pages={257\ndash 286},
         url={https://doi.org/10.1007/s00222-002-0234-y},
      review={\MR{1933586}},
}

\bib{choe-gulliver1992}{article}{
      author={Choe, Jaigyoung},
      author={Gulliver, Robert},
       title={The sharp isoperimetric inequality for minimal surfaces with
  radially connected boundary in hyperbolic space},
        date={1992},
        ISSN={0020-9910},
     journal={Invent. Math.},
      volume={109},
      number={3},
       pages={495\ndash 503},
         url={https://doi-org.prx.library.gatech.edu/10.1007/BF01232035},
      review={\MR{1176200}},
}

\bib{dlmz2000}{article}{
      author={Dalla, L.},
      author={Larman, D.~G.},
      author={Mani-Levitska, P.},
      author={Zong, C.},
       title={The blocking numbers of convex bodies},
        date={2000},
        ISSN={0179-5376},
     journal={Discrete Comput. Geom.},
      volume={24},
      number={2-3},
       pages={267\ndash 277},
         url={https://doi-org.prx.library.gatech.edu/10.1007/s004540010032},
      review={\MR{1758049}},
}

\bib{toth1943}{article}{
      author={Fejes~T\'{o}th, L.},
       title={Covering the sphere with congruent caps},
        date={1943},
     journal={Mat. Fiz. Lapok},
      volume={50},
       pages={40\ndash 46},
}

\bib{toth1964}{book}{
      author={Fejes~T\'{o}th, L.},
       title={Regular figures},
      series={A Pergamon Press Book},
   publisher={The Macmillan Company, New York},
        date={1964},
      review={\MR{0165423}},
}

\bib{gal2013}{incollection}{
      author={Gal, Shmuel},
       title={Search games: a review},
        date={2013},
   booktitle={Search theory},
   publisher={Springer, New York},
       pages={3\ndash 15},
  url={https://doi-org.prx.library.gatech.edu/10.1007/978-1-4614-6825-7_1},
      review={\MR{3087858}},
}

\bib{ghomi:lwr}{article}{
      author={Ghomi, Mohammad},
       title={The length, width, and inradius of space curves},
        date={2018},
        ISSN={0046-5755},
     journal={Geom. Dedicata},
      volume={196},
       pages={123\ndash 143},
  url={https://doi-org.prx.library.gatech.edu/10.1007/s10711-017-0312-3},
      review={\MR{3853631}},
}

\bib{ghomi-wenk-arXiv2020}{article}{
      author={Ghomi, Mohammad},
      author={Wenk, James},
       title={Shortest closed curve to inspect a sphere},
        date={2020},
     journal={arXiv preprint arXiv:2010.15204v1},
}

\bib{ghomi-wenk2021}{article}{
      author={Ghomi, Mohammad},
      author={Wenk, James},
       title={Shortest closed curve to inspect a sphere},
        date={2021},
        ISSN={0075-4102},
     journal={J. Reine Angew. Math.},
      volume={781},
       pages={57\ndash 84},
         url={https://doi.org/10.1515/crelle-2021-0049},
      review={\MR{4343095}},
}

\bib{latala-matlak2017}{incollection}{
      author={Lata\l~a, Rafa\l},
      author={Matlak, Dariusz},
       title={Royen's proof of the {G}aussian correlation inequality},
        date={2017},
   booktitle={Geometric aspects of functional analysis},
      series={Lecture Notes in Math.},
      volume={2169},
   publisher={Springer, Cham},
       pages={265\ndash 275},
      review={\MR{3645127}},
}

\bib{nikonorov2022}{article}{
      author={Nikonorov, Yurii~G.},
       title={Properties of a curve whose convex hull covers a given convex
  body},
        date={2022/09/01},
     journal={Beitr{\"a}ge zur Algebra und Geometrie / Contributions to Algebra
  and Geometry},
      volume={63},
      number={3},
       pages={505\ndash 513},
         url={https://doi.org/10.1007/s13366-021-00613-z},
}

\bib{orourkeMO}{article}{
      author={O'Rourke, Joseph},
       title={Shortest closed curve to inspect a sphere (question posted on
  mathoverflow)},
        date={2011June},
     journal={mathoverflow.net},
        ISSN={questions},
  url={mathoverflow.net/questions/69099/shortest-closed-curve-to-inspect-a-sphere},
}

\bib{royden2014}{article}{
      author={Royen, Thomas},
       title={A simple proof of the {G}aussian correlation conjecture extended
  to some multivariate gamma distributions},
        date={2014},
        ISSN={0972-0863},
     journal={Far East J. Theor. Stat.},
      volume={48},
      number={2},
       pages={139\ndash 145},
      review={\MR{3289621}},
}

\bib{tikhomirov2015}{article}{
      author={Tikhomirov, Konstantin~E.},
       title={On the distance of polytopes with few vertices to the {E}uclidean
  ball},
        date={2015},
        ISSN={0179-5376},
     journal={Discrete Comput. Geom.},
      volume={53},
      number={1},
       pages={173\ndash 181},
  url={https://doi-org.prx.library.gatech.edu/10.1007/s00454-014-9639-9},
      review={\MR{3293493}},
}

\bib{sidak1967}{article}{
      author={\v{S}id\'{a}k, Zbyn\v{e}k},
       title={Rectangular confidence regions for the means of multivariate
  normal distributions},
        date={1967},
        ISSN={0162-1459},
     journal={J. Amer. Statist. Assoc.},
      volume={62},
       pages={626\ndash 633},
  url={http://links.jstor.org/sici?sici=0162-1459(196706)62:318<626:RCRFTM>2.0.CO;2-Z&origin=MSN},
      review={\MR{216666}},
}

\bib{zalgaller:1996}{article}{
      author={Zalgaller, V.~A.},
       title={Extremal problems on the convex hull of a space curve},
        date={1996},
        ISSN={0234-0852},
     journal={Algebra i Analiz},
      volume={8},
      number={3},
       pages={1\ndash 13 (Translation in St. Petersburg Math. J. \textbf{8}
  (1997), no. 3, 369\ndash 379 )},
      review={\MR{1402285}},
}

\end{biblist}
\end{bibdiv}

\end{document}